\documentclass[11pt]{article}

\usepackage[utf8]{inputenc}
\usepackage[T1]{fontenc}
\usepackage{lmodern}
\usepackage{microtype}

\usepackage[margin=1.15in]{geometry}

\usepackage{amsmath,amssymb,amsthm,mathtools}

\usepackage[colorlinks=true,
            linkcolor=blue,
            citecolor=blue,
            urlcolor=blue]{hyperref}
\usepackage[nameinlink,capitalise]{cleveref}
\let\oldeqref\eqref
\renewcommand{\eqref}[1]{\begingroup\hypersetup{linkcolor=black}\oldeqref{#1}\endgroup}
\usepackage{booktabs}
\usepackage{enumitem}

\theoremstyle{plain}
\newtheorem{theorem}{Theorem}[section]

\newtheorem{lemma}[theorem]{Lemma}

\theoremstyle{definition}

\theoremstyle{remark}

\title{Secretary Problem Thresholds and Convergents of $1/e$}

\author{Raúl Sánchez Galán}

\date{\today}

\begin{document}

\maketitle

\begin{abstract}
We prove that if $p/q$ is a continued fraction convergent of $1/e$ with
$q\geq 3$, then, for the secretary problem with $q$ applicants, the optimal
number of initially rejected applicants is $p$.
\end{abstract}

\section{Introduction}

The classical secretary problem is one of the best-known examples in optimal
stopping theory; see, for instance, Gilbert and Mosteller~\cite{GilbertMosteller1966}.
In its standard form, an employer interviews $n$ applicants in a uniformly
random order. The applicants can be ranked unambiguously from best to worst, but
the employer only learns the relative rank of each applicant among those already
interviewed. After each interview, the employer must either accept or reject the
applicant, and this decision is irrevocable. The goal is to maximize the
probability of selecting the best applicant.

A remarkable feature of the problem is that the optimal rule has a very simple
threshold form. One rejects a certain number $r(n)$ of applicants, observing them
only in order to establish a standard, and then accepts the first subsequent
applicant who is better than all applicants seen so far. The associated threshold sequence is recorded in the OEIS as A054404~\cite{OEISA054404}. 
For large $n$, the
optimal value $r(n)$ is asymptotic to $n/e$, and the corresponding probability
of success tends to $1/e$. Thus, Euler's number enters the problem already at
the level of its first-order asymptotics.

The purpose of this paper is to prove a more arithmetically precise connection between the secretary problem and Euler's number. Havil observed in \cite[Sec.~13.13]{Havil2003} that the convergents of the continued fraction of $e^{-1}$ appear to encode exact secretary thresholds: if $p/q$ is such a convergent, then $p$ appears to be the optimal threshold for the secretary problem with $q$ applicants. Equivalently, these convergents appear to be fractions of the form \begin{equation*} \frac{r(n)}{n}. \end{equation*} This phenomenon was also mentioned in~\cite{BayonEtAl2019}, where it was noted that the relation remained open.

Our main result is a short and elementary proof of this observation, summarised in Theorem \ref{main_result}.

We begin by recalling the classical secretary problem; we then review the necessary facts about continued fractions and the continued fraction of $1/e$. The proof of the main
theorem is given in the final section.

\section{The secretary problem}

We consider the following standard version of the secretary or Sultan's dowry problem \cite{GilbertMosteller1966}, \cite{WeissteinSultansDowry}. There are
$n$ applicants to fill a single secretary position. The applicants are interviewed sequentially in random order, with each order being equally likely. The applicants
have a strict total ranking, but after interviewing an applicant, one only knows
the applicant's rank relative to those already interviewed. Immediately after
each interview, the applicant must either be accepted or rejected, and rejected
applicants cannot be recalled. The objective is to maximize the probability of selecting the best applicant overall.

The problem became widely known after Fox and Marnie posed a mathematically
equivalent problem in \cite{FoxMarnie1960}.

The optimal strategy has a threshold form: an applicant who is not the best among those already observed cannot be the best overall, and therefore should never be accepted. Thus only applicants with highest observed rank, i.e., 
\emph{candidates}, need to be considered. If the $i$-th applicant is a candidate, then
the probability that this applicant is the best among all $n$ applicants is
$i/n$, which is increasing in $i$. However, the probability of succeeding by rejecting the present candidate and continuing decreases with $i$, since fewer applicants remain.
Therefore, there is a threshold point: before
it, one rejects even candidates; after it, one accepts the first candidate.

The rule with threshold $0\leq r\leq n-1$ rejects the first
$r$ applicants unconditionally, and then accepts the first applicant who is
better than all previously seen applicants. If no such candidate appears, the last applicant is selected.

Let $P(n,r)$ denote the probability of success for this strategy. For $r=0$, we accept the first applicant, and hence by the uniformity assumption
\begin{equation*}
P(n,0)=1/n.    
\end{equation*}
For $1\leq r\leq n-1$, consider the following events,
\begin{align*}
A_i &= \text{applicant $i$  is selected} \\
B_i &= \text{applicant $i$ has the highest rank overall.}
\end{align*} 
Then 
\begin{equation*}
    P(n,r) = \sum_{i=1}^n  P(A_i \cap B_i) = \sum_{i=1}^n  P(B_i)P(A_i|B_i).
\end{equation*}

Since each order is equally likely, $P(B_i) = \frac{1}{n}$. Additionally,
\begin{equation*}
P(A_i| B_i) = \begin{cases}
			0, & \text{if $i\leq r$}\\
            \frac{r}{i-1}, & \text{if $i > r$}
		 \end{cases}.
\end{equation*}
Indeed, if $i \leq r$, following our strategy the $i$-th applicant will never be selected. 
If $i>r$ and we assume that the $i$-th applicant is the best one, then it will be selected precisely when the previous candidates, in the sense of record applicants, are in the first $r$ positions. Since the relative order of the first $i-1$ applicants is
uniform, this has probability $r/(i-1)$.

Therefore,
\begin{equation}\label{eq:Prob_success}
P(n,r)
=
\sum_{i=r+1}^{n}\frac1n\frac{r}{i-1}
=
\frac{r}{n}\sum_{j=r}^{n-1}\frac1j.
\end{equation}

The optimal threshold can be characterized exactly by comparing consecutive
values of $P(n,r)$. For $1\leq r\leq n-2$,
\begin{equation*}
P(n,r+1)-P(n,r) =
\frac{1}{n}\left(\sum_{j=r+1}^{n-1}\frac1j-1\right).
\end{equation*}
So $P(n,r)$ is increasing with $r$ as long as 
\begin{equation*}
\sum_{j=r+1}^{n-1}\frac1j>1,
\end{equation*}
and then decreases once this sum is at most $1$. 
Consequently, the optimal number
$r(n)$ of initially rejected applicants is characterized by
\begin{equation}\label{eq:criterion}
\sum_{j=r(n)+1}^{n-1}\frac1j\leq 1
\leq
\sum_{j=r(n)}^{n-1}\frac1j.
\end{equation}
The cases $n=1$ and $n=2$ are exceptional: for $n=1$ the only threshold is $r=0$, while for $n=2$ both $r=0$ and $r=1$ are optimal. If the right-hand inequality is strict, the optimum is unique. In fact, except for $n=2$, both inequalities are strict, since it is well-known that the sum of reciprocals of consecutive positive integers cannot be an integer unless the sum consists only of the single term $1$. 
This criterion is the form of the secretary problem that will be used in the proof of the main theorem.

The classical asymptotic rule follows from replacing the harmonic sum by a logarithm. If $r$ and $n$ are both large, then the probability of succeeding \eqref{eq:Prob_success} becomes,
\begin{equation*}
P(n,r)
\sim
\frac{r}{n}\ln\left(\frac{n}{r}\right).
\end{equation*}
Writing $x=r/n$, the asymptotic-limiting function is $
x\ln\left(\frac1x\right)$,
which is maximized at $x=1/e$.
Therefore, in the asymptotic limit, the optimal fraction of initially rejected
applicants tends to $1/e$, that is, $r(n)\sim \frac{n}{e}$, and the optimal probability of success also tends to $1/e$.

\section{Continued fractions}

We recall the basic facts about continued fractions needed to prove the main theorem in the next section. Standard
references include \cite{Khinchin1964,HardyWright2008,Sierpinski}. 

Every irrational real number $x$ has a unique
infinite simple continued fraction expansion
\begin{equation*}
x=[a_0;a_1,a_2,a_3,\dots] := a_0+\cfrac{1}{a_1+\cfrac{1}{a_2+\cfrac{1}{a_3+\cdots}}},
\end{equation*}
where $a_0\in\mathbb Z$ and $a_i\in\mathbb Z_{>0}$ for $i\geq 1$. The
$k$-th convergent of $x$ is the rational number
\begin{equation*}
\frac{p_k}{q_k}
=
[a_0;a_1,\dots,a_k].
\end{equation*}
The numerators and denominators of the convergents satisfy the recurrence
relations
\begin{equation}\label{eq:recurrence}
p_k=a_kp_{k-1}+p_{k-2},
\qquad
q_k=a_kq_{k-1}+q_{k-2},
\end{equation}
with initial conditions $p_{-2}=0$, $p_{-1}=1$,
$q_{-2}=1$, $q_{-1}=0$.

They also satisfy
\begin{equation*}
p_kq_{k-1}-p_{k-1}q_k=(-1)^{k-1},
\end{equation*}
and therefore $p_k/q_k$ is always in lowest terms.

A central property of convergents is that they provide very good rational
approximations. If $p_k/q_k$ and $p_{k+1}/q_{k+1}$ are consecutive convergents
of $x$, then
\begin{equation}\label{eq:rationa_approx}
\left|x-\frac{p_k}{q_k}\right|
<
\frac{1}{q_kq_{k+1}}.
\end{equation}

We shall apply these facts to $x=1/e$. Euler first showed the continued fraction expansion for $e$ in 1737 and published it in~\cite{Euler1744}. For a modern and accessible proof, we refer to~\cite{Cohn2006}. The expansion is
\begin{equation*}
e=[2;1,2,1,1,4,1,1,6,1,1,8,\dots],
\end{equation*}
and therefore,
\begin{equation*}
\frac1e
=
[0;2,1,2,1,1,4,1,1,6,1,1,8,\dots].
\end{equation*}
Thus the first convergents of $1/e$ are
\begin{equation*}
0,\quad
\frac12,\quad
\frac13,\quad
\frac38,\quad
\frac4{11},\quad
\frac7{19},\quad
\frac{32}{87},\quad
\frac{39}{106},\quad
\frac{71}{193},\quad
\frac{465}{1264},\dots.
\end{equation*}

\section{Convergents of \texorpdfstring{$1/e$}{1/e} and the secretary problem}

We shall use the following two elementary estimates.
\begin{lemma}
For every integer $j\geq 1$,
\begin{equation}
\frac1j<\ln\left(\frac{j+\frac13}{j-\frac23}\right),
\qquad
\frac1j>\ln\left(\frac{j+\frac23}{j-\frac13}\right).
\label{eq:log-estimates}
\end{equation}
\end{lemma}

\begin{proof}
    Setting $x=1/j$, the first inequality is equivalent to
\begin{equation*}
x<\ln\left(\frac{1+\frac{x}{3}}{1-\frac{2x}{3}}\right).
\end{equation*}
This inequality, follows from the fact that the function
\begin{equation*}
F(x)=\ln\left(\frac{3+ x}{3-2x}\right)-x
\end{equation*}
satisfies $F(0)=0$ and $F'(x)>0$ for $0<x\leq 1$. 

Similarly, the second inequality follows from considering the function
\begin{equation*}
G(x)=x-\ln\left(\frac{3 + 2x }{3- x }\right),
\end{equation*}
that satisfies $G(0)=0$ and $G'(x)>0$ for $0<x\leq 1$.

\end{proof}

\begin{theorem}\label{main_result}
Let $p/q$ be a convergent of the simple continued fraction of $1/e$ with
$q\geq 3$. In the secretary problem with $q$ applicants, the optimal threshold
$r$, denoting the number of applicants rejected unconditionally at the beginning,
is precisely $p$.
\end{theorem}

\begin{proof}
We have seen that $r$ is the unique optimal threshold if and only if,
\begin{equation}
\sum_{j=r+1}^{n-1}\frac1j<1
\qquad\text{and}\qquad
\sum_{j=r}^{n-1}\frac1j>1.
\label{eq:optimality}
\end{equation}

We shall prove that these two inequalities hold when $n=q$ and $r=p$, for
every convergent $p/q$ of $1/e$ with $q\geq 3$.

Using the first inequality in \eqref{eq:log-estimates}, we get
\begin{equation}
\sum_{j=p+1}^{q-1}\frac1j
<
\sum_{j=p+1}^{q-1}
\ln\left(\frac{j+\frac13}{j-\frac23}\right) = \ln\left(\frac{q-\frac23}{p+\frac13}\right),
\label{eq:upper-harmonic-bound}
\end{equation}
where in the last equality we used that the product $ \prod_{j=p+1}^{q-1}
\left(\frac{j+\frac13}{j-\frac23}\right) $ telescopes.

Similarly,
\begin{equation}
\sum_{j=p}^{q-1}\frac1j
>
\sum_{j=p}^{q-1}
\ln\left(\frac{j+\frac23}{j-\frac13}\right)
=
\ln\left(\frac{q-\frac13}{p-\frac13}\right).
\label{eq:lower-harmonic-bound}
\end{equation}

If $p/q$ is a convergent of $1/e$, and $q'$ is the denominator of the next
convergent, then, using \eqref{eq:rationa_approx},
\begin{equation*}
\left|\frac qe-p\right|<\frac1{q'}.
\label{eq:cf-error}
\end{equation*}
For the convergents of $1/e$ with $q\geq 3$, the next denominator satisfies
$q'\geq 8$. Therefore
\begin{equation*}
- \frac18 < \frac{q}{e} - p <\frac18,
\label{eq:error-eighth}
\end{equation*}
which implies
\begin{equation}
-\frac{e-1}{3e}<\frac qe-p<\frac{e + 2}{3e}.
\label{eq:error-window}
\end{equation}
The right-hand inequality in \eqref{eq:error-window} is equivalent to
\begin{equation*}
\frac{q-\frac23}{p+\frac13}<e,
\end{equation*}
and together with \eqref{eq:upper-harmonic-bound}, gives
\begin{equation*}
\sum_{j=p+1}^{q-1}\frac1j<1.
\end{equation*}
The left-hand inequality in \eqref{eq:error-window} is equivalent to
\begin{equation*}
\frac{q-\frac13}{p-\frac13}>e,
\end{equation*}
and together with \eqref{eq:lower-harmonic-bound}, gives
\begin{equation*}
\sum_{j=p}^{q-1}\frac1j>1.
\end{equation*}
Thus the two optimality inequalities \eqref{eq:optimality} hold for $r=p$ and $n=q$. Hence
$r=p$ is the unique optimal threshold for every convergent $p/q$ of $1/e$
with $q\geq 3$.

\end{proof}

If we interpret the first convergent of $1/e$ as $0/1$, then for every convergent
$p/q$ the numerator $p$ gives an optimal threshold for the secretary problem with
$q$ applicants. The only non-unique case is $q=2$, where both $r=0$ and $r=1$ succeed with probability $1/2$.

\noindent{\small\sc 4i Intelligent Insights, Tecnoincubadora Marie Curie, PCT Cartuja, 41092 Sevilla, Spain}

\noindent E-mail: {\tt \href{mailto:r.sanchez@4i.ai}{r.sanchez@4i.ai}}

\end{document}